\documentclass[11pt]{article}
\usepackage{amsfonts}
\usepackage{amsmath,color}
\usepackage{bm}
\usepackage[makeroom]{cancel}
\usepackage[affil-it]{authblk}

\newcommand{\beq}{\begin{equation}}
\newcommand{\eeq}{\end{equation}}
\newcommand{\bea}{\begin{eqnarray}}
\newcommand{\eea}{\end{eqnarray}}
\newcommand{\beas}{\begin{eqnarray*}}
\newcommand{\eeas}{\end{eqnarray*}}

\newtheorem{theorem}{Theorem}[section]

\newtheorem{definition}[theorem]{Definition}
\newtheorem{proposition}[theorem]{Proposition}

\newtheorem{corollary}[theorem]{Corollary}

\newtheorem{remark}[theorem]{Remark}
\newtheorem{example}[theorem]{Example}
\newtheorem{examples}[theorem]{Examples}
\newtheorem{foo}[theorem]{Remarks}

\newenvironment{Remark}{\begin{remark}\rm}{\end{remark}}

\newtheorem{Prop}{Proposition}
\newtheorem{Lem}{Lemma}

\newenvironment{proof}{\addvspace{\medskipamount}\par\noindent{\it
Proof}.}
{\unskip\nobreak\hfill$\Box$\par\addvspace{\medskipamount}}

\newcommand{\bM}{\mathcal M}

\newcommand{\Dv}{\Delta_\mathcal{V}}
\newcommand{\Ho}{\mathcal H}
\newcommand{\V}{\mathcal V}

\newcommand{\M}{\mathcal M}

\newcommand{\R}{\mathbb R}

\newcommand{\Op}{\ensuremath{\mathcal{L}}}

\newcommand{\ep}{\varepsilon}
\title{Hypocoercive estimates on foliations and velocity spherical Brownian motion}

\author{Fabrice Baudoin%
  \thanks{Author supported in part by Grant NSF-DMS 15-11-328}}
\affil{Department of Mathematics, Purdue University, USA}

\author{Camille Tardif}
\affil{LPMA, Universit\'e Pierre \& Marie Curie, Paris, France}

\date{}

\begin{document}

\maketitle

\begin{abstract}

By further developing the generalized $\Gamma$-calculus for hypoelliptic operators, we prove hypocoercive estimates for a large class of Kolmogorov type operators which are defined on non necessarily totally geodesic Riemannian foliations. We study then in detail the example of  the velocity spherical Brownian motion, whose generator is a step-3 generating hypoelliptic   H\"ormander's type operator. To prove hypocoercivity  in that case, the key point is to show  the existence of a convenient  Riemannian foliation associated to the diffusion. We will then deduce, under suitable geometric conditions, the convergence to equilibrium of the diffusion in $H^1$ and in $L^2$.

\end{abstract}

\tableofcontents

\section{Introduction}

Let $\mathcal{L}$ be a hypoelliptic Kolmogorov type diffusion operator on a smooth and connected manifold $\mathcal{M}$, which admits an invariant probability measure $\mu$. We are interested in the problem of exponential convergence to equilibrium for the semigroup $e^{t\mathcal{L}}$. To address this problem, several tools have been  developed in the last few years. A functional analytic approach, based on previous ideas by Kohn and H\"ormander, relies on  spectral localization tools to prove exponential convergence to equilibrium with explicit bounds on the rate. For this approach, we refer to Eckmann and Hairer \cite{EH}, H\'erau and Nier \cite{HN1}, and  Heffer and Nier \cite{HN2}. On the other side,  L. Wu in \cite{Wu}, Mattingly, Stuart and Higham in \cite{Matt}, Bakry, Cattiaux and Guillin in \cite{BCG}, Talay \cite{Talay}  and F.Y. Wang \cite{FYWang3} use  Lyapunov functions and probabilistic tools to prove  exponential convergence to equilibrium in several cases.

\

A fundamental contribution, closer to the approach of the present paper,  is due to Villani \cite{Villani1}, who introduced in his memoir  the important notion of hypocoercivity (see also  Dolbeault, Mouhot and Schmeiser \cite{Dolbeault}). Villani's theory was since revisited  by using  a generalized Bakry-\'Emery type $\Gamma$-calculus in Baudoin \cite{baudoin-bakry}, Monmarch\'e \cite{mymarket}  and F.Y. Wang \cite{FYWang2}. To study hypocoercivity, complementing methods from Riemannian geometry were also used in Baudoin \cite{BaudoinEMS, baudoinAMS}. 

\

\textbf{Generalized Bakry-\'Emery estimates}

\

The basic idea in Baudoin \cite{baudoin-bakry,BaudoinEMS} is to find a Riemannian metric on $\mathcal{M}$ for which $\mathcal{L}$ satisfies a generalized Bakry-\'Emery estimate. More precisely, let $g$ be a Riemannian metric on $\mathcal{M}$.
Consider then the second order differential bilinear form
\[
\mathcal{T}_2 (f)=\frac{1}{2} ( \mathcal{L} \| \nabla f \|^2 - 2 \langle \nabla f , \nabla \mathcal{L} f \rangle), \quad f \in C_0^\infty(\mathcal{M}),
\]
where $\nabla$ denotes the Riemannian gradient for the metric $g$. Let us denote by $\Gamma$ the \textit{carr\'e du champ operator} of $\mathcal{L}$. It is  proved in in \cite{baudoin-bakry} that if there exist positive constants $K_1,K_2$ such that for every $ f \in C_0^\infty(\mathcal{M})$,
\begin{align}\label{BE}
\mathcal{T}_2 (f)\ge  K_1 \| \nabla f \|^2 -K_2 \Gamma(f),
\end{align}
and that if $\mu$ satisfies a Poincar\'e inequality
\[
\int f^2 d\mu -\left(\int f d\mu\right)^2 \le \frac{1}{\lambda} \int \| \nabla f \|^2 d\mu,
\]
then $e^{t\mathcal{L}}$ converges exponentially fast to equilibrium in $H^1(\mu)$. The rate of convergence may moreover be estimated explcitly in terms of $K_1,K_2$ and $\lambda$ (see \cite{baudoin-bakry}). 

\

Finding  general intrinsic conditions on $\mathcal{L}$ so that there exists a metric $g$ satisfying \eqref{BE} seems to be a very difficult open problem. Theorem 18 in Villani \cite{Villani1} (see also Monmarch\'e \cite{mymarket}) may actually be interpreted as providing such sufficient conditions and gives a way to explicitly construct the metric $g$ in some cases. Another, more geometric, approach is taken in \cite{BaudoinEMS} to construct $g$. In  \cite{BaudoinEMS} , one considers totally geodesic Riemannian foliations whose leaves are determined by the principal symbol of $\mathcal{L}$. The inequality \eqref{BE} is then equivalent to bounds on simple Ricci-like  tensors associated to the foliation.

\

\textbf{Non-totally geodesic foliations}

\

It turns out that for several interesting diffusion operators $\mathcal{L}$, including the generator of the velocity spherical Brownian motion which is thoroughly studied in the present paper,  it does not seem possible to find a totally geodesic foliation satisfying \eqref{BE}.  However, one will exhibit a non totally geodesic foliation for which this works. For this reason, it is interesting to generalize the results of \cite{BaudoinEMS} to the case where the foliation is not totally geodesic. This is what we do in Section 2 of the present paper. In our main result, we give a tensorial expression of $\mathcal{T}_2$, from which one can easily deduce sufficient conditions ensuring that \eqref{BE} is satisfied. In the non-totally geodesic case, the Bochner's type identities are much more involved and an important tool to prove the tensorial expression of $\mathcal{T}_2$ is to work with a connection introduced by Hladky in \cite{Hladky}.

\

\textbf{Velocity spherical Brownian motion}

\

In Section 3 we study the convergence to equilibrium, and prove hypocoercivity, for a class of diffusions called \emph{Velocity Spherical Brownian Motions}. Results of Section 2 can not be applied directly, because the main problem here is precisely to find a metric for which a generalized Bakry-\'Emery estimate is satisfied.  The velocity spherical Brownian motion is a diffusion process valued in $T^1\mathcal{M}$, the unit tangent bundle of a Riemannian $n$-dimensional manifold $\mathcal{M}$ of finite volume, and was originally introduced in Angst-Bailleul-Tardif  \cite{ABT} under the name kinetic Brownian motion. It is a velocity/position process ($T^1\mathcal{M}$ is seen as a phase space) where the velocities live on the tangent sphere and have Brownian dynamics. The generator takes the form 
\[
\Op:= \frac{\sigma^2}{2} \Delta^{v} + \kappa \xi,
\]
where $\sigma$ and $\kappa$ are parameters, $\Delta^{v}$ is the vertical  Laplacian on the spherical fibers and $\xi$ is the vector field of the unit tangent bundle generating the geodesic flow. The diffusion $\Op$ is a hypoelliptic Kolmogorov type diffusion operator on the manifold $T^1\mathcal{M}$. For example, when $\mathcal{M}=\R^n$ and $\sigma=\kappa=1$, it may be identified with $\left(\theta_t, \int_{0}^t \theta_s ds\right)_{t\geq 0}$ where $(\theta_t)_{t\geq 0}$ is a Brownian motion on the $(n-1)$-dimensional sphere seen as a submanifold of $\R^n$. The velocity spherical Brownian motion may be seen as the Riemannian counterpart of the relativistic diffusion introduced by Franchi and  Le Jan in \cite{FLJ07}.  An interesting property of the motion  is that it interpolates between the geodesics (when $\sigma =0$, $\kappa=1$) and the Brownian motion on $\mathcal{M}$ (when $\kappa=\sigma$ goes to infinity), see \cite{ABT} but also \cite{XMLi} for this homogenization result. It is also close to Bismut's Hypoelliptic Laplacian \cite{Bismut}, which lives on $T\mathcal{M}$ and where the velocities have the dynamics of an Ornstein-Uhlenbeck process on the fiber instead of spherical Brownian motion one. The velocity spherical Brownian motion is also a natural Riemannian generalization,  with zero-potential, of the spherical Langevin process which is studied for example in \cite{GS2013}, \cite{GKMSW2014} and \cite{GS2014} (see also references therein). The latter process arises in industrial applications as the so-called fiber lay-down process in modeling virtual nonwoven webs. 

\

The hypocoercivity of the spherical Langevin process in $\mathbb{R}^n$ has been proved in \cite{GS2014} extending abstract Hilbert space strategy developed by Dolbeault, Mouhot and Schmeiser in \cite{Dolbeault}. Nevertheless, quoting \cite{GS2014} \textit{ "it is an open problem to apply the methods from Villani's memoir \cite{Villani1} to the spherical velocity Langevin equation"}. Indeed, as remarked in \cite{GS2014} , Villani's approach seems to be, a priori, not adapted due to the geometry of Lie brackets relations between vector fields occurring in the decomposition of $\Op$. More precisely, we look at a generator $\Op$ which is locally of the form $\sum_{i=1}^{n-1}A_i^{*}A_{i} + B$ where to get all the directions in the tangent  space, (to get H\"ormander condition of hypoellipticity) we need to consider $A_i$, $[A_i,B]$ for  $i=1,\dots n-1$ and the last direction $B$ is obtain via the $3$-order bracket $[A_i, [A_i, B] ]$ (for some $i$). It is not possible to get all the directions considering only brackets of the form $[ \cdots[C_2,[C_1, B] B ] \cdots B]$, as it is needed with Villani's method (see his memoir for the notations). 

\

Nevertheless, we will be able to extend Villani's methodology  to construct a metric $g$ on $T^1\mathcal{M}$ for which a generalized Bakry-\'Emery estimate is satisfied. It is worth noting that the metric $g$, though equivalent to the Sasaki metric on $T^1\mathcal{M}$, is associated to a non-totally geodesic Riemannian foliation on $T^1\mathcal{M}$,  as studied in Section 2.  We deduce then the following main  theorem of the paper.

\begin{theorem}
Let
\[
\Op= \frac{\sigma^2}{2} \Delta^{v} + \kappa \xi,
\]
be the generator of the velocity spherical Brownian motion on $T^1\mathcal{M}$. Let us assume that the Riemannian curvature tensor of $\mathcal{M}$ is bounded and that $\mathcal{M}$ is complete. Let us moreover assume that the normalized  Sasaki-Riemannian measure $\mu$ on  $T^1\mathcal{M}$ satisfies a Poincar\'e inequality
\[
\int f^2 d\mu -\left(\int f d\mu\right)^2 \le \frac{1}{\lambda} \int \| \nabla f \|^2 d\mu.
\]
Then, there exist $C_1,C_2>0$, such that for every $t \ge 0$, and $f \in H^1(\mu)$,
\[
\left\| e^{t\Op} f - \int f d\mu  \right\|_{H^1(\mu)} \le C_1 e^{-C_2 t} \left\| f- \int f d\mu \right\|_{H^1(\mu)},
\]
where 
\[
\| f \|^2_{H^1(\mu)}=\int f^2 d\mu + \int \| \nabla f \|^2 d\mu,
\]
and $\nabla$ is the Riemannian gradient for the standard Sasaki metric on $T^1\mathcal{M}$.
Moreover, the convergence also holds in $L^2(\mu)$, i.e.  there exist $C_3 >0$, such that for every $t \ge 0$, and $f \in L^2(\mu)$,
\[
\left\| e^{t\Op} f - \int f d\mu  \right\|_{L^2(\mu)} \le C_3 e^{-C_2 t} \left\| f - \int f d\mu \right\|_{L^2(\mu)}.
\]
\end{theorem}

To get the convergence in $L^2$ norm we prove an interesting regularization result (see Lemma \ref{regularization} below) by adapting H\'erau's method (see \cite{Herau}) to our  case. Once again, computations are more tedious due to the particular geometry of the Lie brackets and our regularization result is certainly not optimal.

\section{Kolmogorov type operators on Riemannian foliations}

\subsection{Generalized $\Gamma$-calculus for Kolmogorov type operators}

Let $\M$ be a smooth, connected  manifold with dimension $n+m$. We assume that $\bM$ is equipped with a Riemannian foliation with  $m$-dimensional leaves. We denote by $\Dv$ the vertical Laplacian of the foliation, $\nabla_\V$ the vertical gradient and $\nabla_\Ho$ the horizontal gradient.

\begin{definition}
We call Kolmogorov type operator a hypoelliptic diffusion operator $\mathcal L$ on $\M$ that can be written as
\[
\mathcal L=\Dv+Y,
\]
where $Y$ is a smooth vector field on $\M$.
\end{definition}

The geometry of a Riemannian foliation can locally be described in local orthonormal frames.  A local orthonormal frame of smooth vector fields $\{ X_1, \cdots ,X_n, Z_1, \cdots , Z_m \}$ is said to be adapted if the $X_i$'s are horizontal and the $Z_l$'s vertical, that is tangent to the leaves. Since the leaves are integral sub-manifolds,  one can write the structure constants as follows:
\begin{align}\label{frame}
\begin{cases}
[Z_\gamma,Z_\beta]=\sum_{\alpha} \omega_{\gamma \beta}^\alpha Z_\alpha \\
[Z_\beta,X_i]=\sum_{k}\omega_{\beta i}^k X_k +\sum_{\alpha} \omega_{\beta i}^\alpha Z_\alpha.
\end{cases}
\end{align}
We will always stick to the convention that the  letter $Z$ is reserved for vertical fields and the letter $X$ for horizontal fields. The greek indices will be for summations on vertical directions and the latin indices for summations on horizontal directions. We observe that the Riemannian foliation is bundle-like (see \cite{Tondeur} page 56) if and only if
\[
\omega_{\beta i}^k=-\omega_{\beta k}^i ,
\]
and moreover totally geodesic (see \cite{Tondeur} page 58) if and only if moreover
\[
\omega_{\beta i}^\alpha=-\omega_{\alpha i}^\beta.
\]
In such a frame the vertical Laplacian is given by
\[
\Dv=\sum_\alpha Z_\alpha^2 -\sum_{\alpha ,\beta} \omega_{\alpha \beta}^\beta Z_\alpha.
\]

\

In the sequel of the section, we consider a Kolmogorov type operator 
\[
\mathcal L=\Dv+Y.
\]

In this general framework, to study $\mathcal L$, it will be more convenient not to work with the  Levi-Civita connection of the Riemannian metric, but with a metric  connection for which the horizontal and vertical bundles are parallel. Such a connection was introduced by Hladky in \cite{Hladky}.

The Hladky connection of the foliation is a metric connection $\nabla$ with torsion tensor $T$ such that:

\begin{itemize}
\item If $U\in \Gamma^\infty(\mathcal{H})$, then $\nabla U \in \Gamma^\infty(\mathcal{H})$;
\item If $V \in \Gamma^\infty(\mathcal{V})$, then $\nabla V \in \Gamma^\infty(\mathcal{V})$;
\item If  $U,V \in \Gamma^\infty(\mathcal{V})$, then $T(U,V)=0$;
\item If $U,V \in \Gamma^\infty(\mathcal{H})$, then $T(U,V)=-[U,V]_\V$, where $[U,V]_\V$ is the vertical part of $[U,V]$;
\end{itemize}

One can check that

\[
\nabla_U V =
\begin{cases}
 ( D_U V)_{\mathcal{H}} , \quad U,V \in \Gamma^\infty(\mathcal{H}) \\
 ( D_U V)_{\mathcal{V}}, \quad U,V \in \Gamma^\infty(\mathcal{V})
\end{cases}
\]
where $D$ is the Levi-Civita  connection of the Riemannian metric and the subscript $\Ho$  (resp. $\mathcal{V}$) denotes the projection on $\mathcal{H}$ (resp. $\mathcal{V}$). Actually, in the local frame \eqref{frame}, one has the following formulas (see Example 2.16 in \cite{Hladky}):
\begin{align}\label{formulas}
\begin{cases}
\nabla_{Z_\beta} X_j =\frac{1}{2} \sum_k (\omega_{\beta j}^k -\omega_{\beta k}^j)X_k \\
\nabla_{X_i} Z_\beta=-\frac{1}{2} \sum_\alpha ( \omega_{ \beta i}^\alpha-\omega_{ \alpha i}^\beta)Z_\alpha \\
\nabla_{Z_\gamma} Z_\beta= \frac{1}{2} \sum_\alpha (\omega_{\gamma \beta}^\alpha +\omega_{\alpha \gamma}^\beta+\omega_{\alpha \beta}^\gamma) Z_\alpha \\
T( Z_\beta, X_j)=-\frac{1}{2} \sum_k ( \omega_{\beta k }^j + \omega_{\beta j }^k) X_k -\frac{1}{2} \sum_\alpha  ( \omega_{\alpha j }^\beta + \omega_{\beta j }^\alpha) Z_\alpha .
\end{cases}
\end{align}

Associated to $\mathcal L$, we consider the Bakry's $\Gamma_2$ operator which is defined, for $f,g \in C_0^\infty(\M)$ by
\[
\Gamma_2(f,g)=\frac{1}{2} ( \mathcal L\Gamma(f,g)-\Gamma(g, \mathcal Lf)-\Gamma(f,\mathcal Lg)),
\]
where
\[
\Gamma(f,g) =\langle \nabla_\V f,\nabla_\V g \rangle.
\]
For $f \in C_0^\infty(\M)$, we will simply denote $\Gamma (f,f)$ and $\Gamma_2(f,f)$ respectively by $\Gamma(f)$ and $\Gamma_2(f)$.

\

Our first result is a Bochner's type identity for $\mathcal L$. In order to state it, we introduce some tensors associated to the  connection $\nabla$. 

If $f$ is a smooth function, the vertical Hessian of $f$ will be denoted by  $\nabla^2_\V f$ and is defined on vertical vectors by
\[
\nabla^2_\V f (U,V)=\langle \nabla_U \nabla_\V f ,V\rangle, \quad U,V \in \Gamma^\infty(\mathcal{V}).
\]
If $U$ or $V$ is horizontal, we define $\nabla^2_\V f (U,V)=0$.  It is easy to check that $\nabla^2_\V f $ is symmetric because for $ U,V \in \Gamma^\infty(\mathcal{V})$, $T(U,V)=0$.

The  Ricci curvature of the connection $\nabla$ will denoted by $\mathbf{Ric}$ and, as usual, $\mathbf{Ric}(U,V)$ is defined as the trace of the endomorphism $W \to R(W,X)Y$ where $R$ is the Riemann curvature tensor of $\nabla$.

Finally, the tensor $\nabla Y$ is defined for $U,V \in \Gamma^\infty(T \M)$ by 
\[
\nabla Y (U,V)=\langle \nabla_U Y,V \rangle.
\]

\begin{proposition}\label{Ga_2}
For $f \in C^\infty_0(\M)$,
\[
\Gamma_2(f)=\| \nabla^2_\V f \|^2 + \mathbf{Ric} (\nabla_\V f , \nabla_\V f) -\nabla Y(\nabla_\V f , \nabla f)-\langle T(Y,\nabla_\V f), \nabla f \rangle
\]
where $\| \nabla^2_\V f \|^2$ is the Hilbert-Schmidt norm of the vertical Hessian.
\end{proposition}

\begin{proof}
We split $\Gamma_2$ in two parts.  Let us first observe that from the usual Bochner's formula in Riemannian geometry, we have
\[
\frac{1}{2} (\Dv \| \nabla_V f\|^2- 2 \langle \nabla_\V f , \nabla_\V \Dv f \rangle )=\| \nabla^2_\V f \|^2 + \mathbf{Ric} (\nabla_\V f , \nabla_\V f).
\]
We then compute
\[
\frac{1}{2} (Y \| \nabla_\V f\|^2- 2 \langle \nabla_\V f , \nabla_\V Y f \rangle ),
\]
by introducing a local vertical orthonormal frame $Z_1,\cdots,Z_m$. In this frame we have
\begin{align*}
\frac{1}{2} (Y \| \nabla_\V f\|^2- 2 \langle \nabla_\V f , \nabla_\V Y f \rangle )&= \sum_{\alpha } (YZ_\alpha f)Z_\alpha f -\sum_{\alpha} (Z_\alpha Y f) Z_\alpha f \\
 & =\sum_{\alpha } [Y,Z_\alpha ] fZ_\alpha f \\
 &=\sum_{\alpha} (\nabla_Y Z_\alpha -\nabla_{Z_\alpha} Y -T(Y,Z_\alpha) )fZ_\alpha f 
\end{align*}
Since the covariant derivative of vertical fields is vertical and the connection $\nabla$ is metric, we have
\[
\sum_{\alpha} (\nabla_Y Z_\alpha )fZ_\alpha f =\sum_{\alpha,\beta} \langle \nabla_Y Z_\alpha ,Z_\beta \rangle Z_\alpha fZ_\beta f =0.
\]
The proof is then completed by putting the two pieces together.
\end{proof}

 We now define for $f,g \in C_0^\infty (\M)$,
 \[
\Gamma^\Ho_2(f,g)=\frac{1}{2} ( \mathcal L\langle \nabla_\Ho f,\nabla_\Ho g\rangle -\langle  \nabla_\Ho g, \nabla_\Ho  \mathcal Lf\rangle - \langle \nabla_\Ho f,\nabla_\Ho \mathcal Lg\rangle ).
\]

As before, as a shorthand notation, we will denote $\Gamma^\Ho_2(f):=\Gamma^\Ho_2(f,f)$. Before we proceed to the Bochner's identity for $\Gamma^\Ho_2$, let us introduce the relevant tensors. We define the following tensors for $f \in C_\infty(\M)$, $U \in \Gamma^\infty(\M)$, in the frame \eqref{frame}:

\[
\| \nabla^2_{\Ho,\V} f \|^2 =\sum_{i,\alpha} \left(  \langle \nabla_{Z_\alpha} \nabla f , X_i \rangle-T(Z_\alpha,X_i)_\Ho  \right)^2,
\]

\[
\tau(U)=-\sum_\alpha \nabla_{Z_\alpha} T (Z_\alpha,U)-\sum_\alpha T(Z_\alpha,T(Z_\alpha,U))-\sum_\alpha T(Z_\alpha,T(Z_\alpha,U))_\Ho.
\]

and 

\[
\Theta (U)=\sum_{\alpha ,\beta} \langle T(Z_\alpha , U), Z_\beta \rangle Z_\alpha^* \otimes Z_\beta^*.
\]

\begin{proposition}\label{Ga3}
For $f \in C_0^\infty(\M)$.
\begin{align*}
\Gamma^\Ho_2(f) =& \| \nabla^2_{\Ho,\V} f \|^2+2 \langle  \nabla^2_{\V} f , \Theta (\nabla_\Ho f) \rangle+ \mathbf{Ric}(\nabla_\Ho f , \nabla_\V f)+ \langle \tau (\nabla_\Ho f) , \nabla f \rangle \\
 &-\nabla Y(\nabla_\Ho f , \nabla f)-\langle T(Y,\nabla_\Ho f), \nabla f \rangle.
\end{align*}
where $\langle  \nabla^2_{\V} f , \Theta (\nabla_\Ho f) \rangle$ denotes the Hilbert-Schmidt inner product.
\end{proposition}

\begin{proof}
We shall prove this identity at the center of the local frame \eqref{frame}. It is easy to see, working with normal coordinates on the leaves, that we can assume that at the center of the frame $\omega_{\alpha \beta}^\gamma =0$. Using the chain rule we  see that
\[
\Gamma^\Ho_2(f)= \sum_{\alpha,i} \left( Z_\alpha X_i f \right)^2 +\sum_i [\Delta_\V ,X_i] f X_i f +\sum_i [Y,X_i] f X_i f.
\]
Repeating the argument of the proof of Proposition \ref{Ga_2}, we have
\[
\sum_i [Y,X_i] f X_i f=-\nabla Y(\nabla_\Ho f , \nabla f)-\langle T(Y,\nabla_\Ho f), \nabla f \rangle.
\]
We now compute that at the center of the frame
\begin{align*}
\sum_i [\Delta_\V ,X_i] f X_i f & =\sum_{\alpha,i} [Z_\alpha^2 ,X_i]f X_i f -\sum_{i,\alpha,\beta} [ \omega_{\alpha \beta}^\beta Z_\alpha,X_i ]f X_i f \\
 &=2 \sum_{\alpha,i} Z_\alpha [Z_\alpha ,X_i]f X_if -\sum_{\alpha,i} [Z_\alpha, [Z_\alpha ,X_i]]f X_if +\sum_{i,\alpha,\beta} (X_i \omega_{\alpha \beta}^\beta)  Z_\alpha f X_i f .
\end{align*}
Expanding the previous inequality by using the structure constants and completing then the squares, yields
\begin{align*}
 & \sum_{\alpha,j} \left( Z_\alpha X_j f \right)^2 +\sum_i [\Delta_\V ,X_i] f X_i f \\
=&  \sum_{\alpha,j} \left( Z_\alpha X_j f +\sum_i \omega_{\alpha i}^j X_i f \right)^2 +2 \sum_{i,\alpha,\beta} \omega_{\alpha i}^\beta Z_\alpha Z_\beta f X_if \\
 &  +\sum_{\beta,i} \left(\sum_\alpha Z_\alpha \omega_{\alpha i}^\beta -\sum_{\alpha , j} \omega_{\alpha i}^j \omega_{\alpha j}^\beta + \sum_\alpha X_i \omega_{\beta \alpha}^\alpha   \right) X_i f Z_\beta f \\
 & +\sum_{i,j} \left(\sum_\alpha Z_\alpha \omega_{\alpha i}^j -\sum_{\alpha , l} \omega_{\alpha i}^l  ( \omega_{\alpha l}^j   +\omega_{\alpha j}^l ) \right) X_i f X_j f.
\end{align*}

It is then a direct, but tedious, exercise to check that

\begin{align*}
 & \sum_{\beta,i} \left(\sum_\alpha Z_\alpha \omega_{\alpha i}^\beta -\sum_{\alpha , j} \omega_{\alpha i}^j \omega_{\alpha j}^\beta + \sum_\alpha X_i \omega_{\beta \alpha}^\alpha   \right) X_i f Z_\beta f 
 +\sum_{i,j} \left(\sum_\alpha Z_\alpha \omega_{\alpha i}^j -\sum_{\alpha , l} \omega_{\alpha i}^l  ( \omega_{\alpha l}^j   +\omega_{\alpha j}^l ) \right) X_i f X_j f \\
 =& \mathbf{Ric}(\nabla_\Ho f , \nabla_\V f)-\langle \sum_\beta \nabla_{Z_\beta} T (Z_\beta , \nabla_\Ho f) , \nabla f \rangle  
  -\langle  \sum_\beta T(Z_\beta, T(Z_\beta, \nabla_\Ho f)), \nabla f \rangle \\
   & -\langle  \sum_\beta T(Z_\beta, T(Z_\beta, \nabla_\Ho f)), \nabla_\Ho f \rangle.
\end{align*}
\end{proof}

We observe that the computations considerably simplify if the foliation is bundle like and totally geodesic (close computations in that case have already  been done in \cite{BaudoinEMS} Theorem 7.2).

\begin{corollary}
Assume that the foliation is bundle like and totally geodesic, then for $f \in C_0^\infty(\M)$.
\begin{align*}
\Gamma^\Ho_2(f) =& \| \nabla^2_{\Ho,\V} f \|^2-\nabla Y(\nabla_\Ho f , \nabla f)-\langle T(Y,\nabla_\Ho f), \nabla f \rangle.
\end{align*}
 \end{corollary}
 
 \begin{proof}
 When  the foliation is bundle like and totally geodesic, we can assume that the foliation comes from a totally geodesic submersion. The horizontal $X_i$ can then be chosen to be formed of basic vector fields.  As we have seen above, we have
 \[
\Gamma^\Ho_2(f)= \sum_{\alpha,i} \left( Z_\alpha X_i f \right)^2 +\sum_i [\Delta_\V ,X_i] f X_i f +\sum_i [Y,X_i] f X_i f.
\]
From \cite{BeBo}, we have $[\Delta_\V ,X_i] =0$, thus
 \[
\Gamma^\Ho_2(f)= \sum_{\alpha,i} \left( Z_\alpha X_i f \right)^2 +\sum_i [Y,X_i] f X_i f,
\]
and the conclusion easily follows.
 \end{proof}

\subsection{Gradient bounds for the semigroup}

With the computations of $\Gamma_2$ and $\Gamma_2^\Ho$ in hands, we can now argue as in \cite{baudoin-bakry, BaudoinEMS} to deduce criteria for hypocoercivity, that is a quantitative convergence to equilibrium for the semigroup generated by $\mathcal L$. For the sake of completeness, we recall the main results.

\

An analytic difficulty that arises when studying Kolmogorov type operators is that, in general, they are not symmetric with respect to any measure. As a consequence, we can not use functional analysis and the spectral theory of self-adjoint operators to define the semigroup generated by $L$. A typical assumption to ensure that $L$ generates a well-behaved semigroup is the existence of a nice Lyapunov function.  So, in the sequel,  we will  assume that there exists a function $W$ such that $W \ge 1$, $\| \nabla W \| \le C W$, $LW \le C W$ for some constant $C>0$ and $\{ W \le m \}$ is compact for every $m$.  The assumption about the existence of this function $W$ such that $LW \le CW$  implies that $L$ is the generator of a Markov semigroup $(P_t)_{t \ge 0}$ that uniquely solves the heat equation in $L^\infty$. 

 If $f \in C^\infty(\M)$, we denote
\[
\mathcal{T}_2(f)=\frac{1}{2} \left( \mathcal L (\| \nabla f \|^2) -2\langle \nabla f , \nabla \mathcal L f \rangle \right),
\]
where $\nabla$ is the whole Riemannian gradient, that is $\nabla=\nabla_\Ho+\nabla_\V$. From Propositions \ref{Ga_2} and \ref{Ga3}, we have
\begin{align*}
\mathcal{T}_2(f) = & \| \nabla^2_{\V} f \|^2+ \| \nabla^2_{\Ho,\V} f \|^2+\| \nabla^2_\V f \|^2+2 \langle  \nabla^2_{\V} f , \Theta (\nabla_\Ho f) \rangle+ \mathbf{Ric}(\nabla f , \nabla_\V f) \\
 & + \langle \tau (\nabla_\Ho f) , \nabla f \rangle -\nabla Y(\nabla f , \nabla f)-\langle T(Y,\nabla f), \nabla f \rangle .
 \end{align*}
 
 This implies
 \begin{align*}
\mathcal{T}_2(f) \ge & \mathbf{Ric}(\nabla f , \nabla_\V f)- \| \Theta (\nabla_\Ho f) \|^2+ \langle \tau (\nabla_\Ho f) , \nabla f \rangle -\nabla Y(\nabla f , \nabla f)  -\langle T(Y,\nabla f), \nabla f \rangle.
\end{align*}

\begin{proposition}\label{GB18}
Let us assume that for some $K \in \R$,
\[
\mathcal{T}_2(f) \ge -K \| \nabla f \|^2,
\]
then for every bounded and Lipchitz function $f \in C^\infty(\M)$, we have for $t \ge 0$
\[
\| \nabla P_t f \|^2 \le e^{2K t} P_t (\| \nabla f \|^2).
\]
\end{proposition}

\begin{proof}
This is a consequence of \cite{FYWang2}  (see also \cite{BaudoinEMS}).
\end{proof}

\subsection{Convergence to equilibrium in $H^1$}

As  in Theorem 7.6 in \cite{BaudoinEMS}, which only concerned the bundle like and totally geodesic case, we therefore obtain:

\begin{corollary}\label{poinca}
Assume that there exist two constants $\rho_1 \ge 0$, $\rho_2 >0$ such that for every $f \in C_0^\infty(\M)$,
\begin{align}\label{pointwise}
\mathcal{T}_2(f) \ge -\rho_1 \| \nabla_\V f \|^2 +\rho_2 \| \nabla_\Ho f \|^2.
\end{align}

Assume moreover that  the operator $\mathcal L$ admits  an invariant  probability measure $\mu$ that satisfies the Poincar\'e inequality
\[
\int_{\M}\| \nabla f \|^2 d\mu \ge \kappa \left[ \int_{\M} f^2 d\mu -\left( \int_{\M} f d\mu\right)^2 \right].
\]
Then, for every bounded and Lipschitz function $f$ such that  $\int_{\M} f d\mu=0$,
\begin{align*}
  (\rho_1+\rho_2)\int_{\M} (P_t f)^2d\mu +\int_{\M} \| \nabla P_t f \|^2d\mu  
 \le   e^{-\lambda t}\left(   (\rho_1+\rho_2) \int_{\M} f^2d\mu + \int_{\M} \| \nabla f \|^2 d\mu \right) ,
\end{align*}
where $\lambda=\frac{2\rho_2 \kappa}{\kappa+\rho_1+\rho_2}$.
\end{corollary}

\section{Hypocoercivity of the velocity spherical Brownian motion}
\subsection{Introduction}

The so-called velocity spherical Brownian motion on the unit tangent bundle of a Riemannian manifold $T^{1} \mathcal{M}$ is introduced in  \cite{ABT} (where it is called the kinetic Brownian motion). It is a two-parameters family of hypoelliptic diffusion on $T^1 \mathcal{M}$ which is a perturbation of the geodesic flow by a vertical Laplacian on the fibers. The velocity spherical Brownian motion is a Kolmogorov type process  which is similar to the Langevin process. The difference being that the velocities have Brownian dynamics on the compact fibers (tangent sphere) whereas for the Langevin process,  the velocities have Orstein-Uhlenbeck dynamics on the (non-compact) tangent space. It is shown in \cite{ABT} that when the parameters go to infinity, the process projected on the base Riemmanian manifold converges in law to a Brownian motion. In this section we obtain, under the condition that the Riemannian tensor of the base manifold is bounded, that the velocity spherical Brownian motion  converges in $H^1$ and in $L^2$, when $t$ goes to infinity,  to the equilibrium measure (the renormalized Riemannian volume on $T^1(\mathcal{M})$) at some exponential rate. This rate can be expressed explicitly in terms of the parameters. It is expected that the optimal rate converges, when the parameters go to infinity, to the spectral gap of the base manifold, but unfortunately the rate obtained converges to $0$ and we do not reach the base manifold spectral gap. 

The idea to obtain this exponential rate of convergence is to find a definite-positive quadratic tensor $\mathcal{T}$ on $T^1 \mathcal{M}$ such that there is a  generalized Bakry-\'Emery rstimate
\[
\mathcal{T}_2 \geq \rho \mathcal{T} - K \Gamma.
\]
where $\rho$ is a strictly positive constant. We obtain this local inequality under the condition that the Riemann tensor of $\mathcal{M}$ is bounded. 
 
As explained in the previous section (see also \cite{baudoin-bakry}, \cite{baudoinAMS}) this inequality, together with a Poincar\'e inequality on $T^{1} \mathcal{M}$ provides the exponential convergence to equilibrium in $H^1$ norm. This scheme is close to Talay article \cite{Talay} and Villani's book \cite{Villani1}. Nevertheless in our case the bracket condition of Villani's Theorem is not fulfilled and one cannot apply directly his result. 

To obtain convergence in $L^2$ norm we prove some regularization estimates using methods inspired by H\'erau's work \cite{Herau}. Again, comparing to the case of the kinetic Fokker-Planck equation, our computations are more complicated because of the much more intricate Lie algebra structure for the generator.

\subsection{Geometric framework}

Let $\mathcal{M}$ be a connected, oriented and complete Riemannian  $n$-manifold. Denote by $O\mathcal{M}$ the orthonormal frame bundle. For  $x\in \mathcal{M}$ one denotes by $\mathbf{e} = (e^0, e^1,\dots,e^{n-1})$ an orthonormal frame in $T_{x} \mathcal{M}$. Denoting by $T^1\mathcal{M}$ the unit tangent bundle one defines the bundle projection $\pi$ by 
\begin{align*}
\pi : \  & O\mathcal{M}  \longrightarrow  T^1 \mathcal{M} \\ 
& (x, \mathbf{e})  \longmapsto  (x, e^0).
\end{align*}
Given a local chart $(x_i)_{i=0,\dots, n-1}$ in $\mathcal{M}$ and denoting by $(e_k^{j})_{k}$ the coordinates of $e^{j}$ one defines  for $i=1,\dots, n-1$ the \emph{vertical} vector fields on $O\mathcal{M}$ by
\[
V_i := \sum_{k=0}^{n-1} e_k^{i} \partial_{e_k^{0}} -e_k^{0} \partial_{e_k^{i}}.
\]
Define also the \emph{horizontal} vector fields for $i_0=0,\dots, n-1$ by
\[
H_{i_0} :=   e^{i_0}_{k} \partial_{x_k} - \Gamma^{l}_{ij} e^{k}_i e^{i_0}_j \partial_{e^{k}_l}.
\]
The generator $\Op:= \frac{\sigma^2}{2} \Delta^{v} + \kappa \xi$ of the velocity spherical Brownian motion $T^1 \mathcal{M}$ is the projection of 
\[
\tilde{\Op} := \frac{\sigma^2}{2} \sum_{i=1}^{n-1} V_i^{2} + \kappa H_0,
\]
in the sense that 
\[
\forall f \in C^2(T^1 \mathcal{M}), \quad \tilde{\Op} (f \circ \pi) = (\Op f) \circ \pi.
\]
Here are the fundamental relations involving Lie brackets of the vertical and horizontal vectors fields, ensuring that $\Op$ satisfies H\"ormander condition and is therefore hypoelliptic: 
\[
[V_i, H_0 ] = H_i , \quad [V_j , [V_i , H_0 ] ] =- \delta_{ij} H_0, \quad [H_0, H_i] =  \sum_{j=1}^{n-1}\langle R( e^{i}, e^{0}) e^{0}, e^{j} \rangle V_j
\]
where $R$ is the Riemannian curvature tensor on $\mathcal{M}$ and $\langle \cdot, \cdot \rangle$ is the metric on $\mathcal{M}$. 

\subsection{$\Gamma$-calculus for the velocity spherical Brownian motion}

The tangent space of $T^1\mathcal{M}$ splits up in a direct sum of a vertical part (generated by $\pi^* V_i$, $i=1,\dots, n-1$ ) and the horizontal part (generated by $\pi^* H_i$, $i=0, \dots, n-1$). This horizontal part split up in the direction $\xi= \pi^* H_0$ and its orthogonal (for the Sasaki metric) generated by $\pi^* H_i$, for $i=1,\dots, n-1$ which will be denoted $\tilde{\mathcal{H}}$. So 
\[
T(T^1 \mathcal{M}) = \mathcal{V} \oplus \tilde{\mathcal{H}} \oplus \mathrm{Vect}(\xi).
\]
We introduce for each of those subspace the Gamma operator associated ($\Gamma^{v}$, $\Gamma^{\tilde{h}}$ and $\Gamma^{\xi}$) and the mixed Gamma $\Gamma^{v,\tilde{h}}$. For smooth functions $f,g : T^1 \mathcal{M} \to \R$ define
\begin{align*}
\Gamma^{v}(f,g) &:= \sum_{i=1}^{n-1} V_i (f \circ \pi )  V_i(g \circ \pi)  \\
\Gamma^{v,\tilde{h}}(f,g) &:=\frac{1}{2} \sum_{i=1}^{n-1} H_i (f \circ \pi)  V_i (g \circ \pi) + H_i (g \circ \pi)  V_i (f \circ \pi) \\ 
\Gamma^{\tilde{h}} (f,g) &:= \sum_{i=1}^{n-1} H_i (f\circ \pi)  H_i (g \circ \pi)  \\ 
\Gamma^{\xi}(f,g) &:= H_0(f \circ \pi)  H_0 (g \circ \pi).
\end{align*}
And define also for any $\Gamma$ the corresponding $\Gamma_2$ by the formula 
\[
\Gamma_2 (f,g) := \frac{1}{2} \left ( \Op ( \Gamma(f,g) ) - \Gamma(\Op f , g) -\Gamma(f, \Op g) \right ).
\]
Denote also $\Gamma(f,f)$ by $\Gamma(f)$. 

The vertical and horizontal gradient of a function $f: T^1 \mathcal{M} \to \R$ are given by: 
\begin{align*}
\nabla^{v} f (x, \mathbf{e}) &:=  \sum_{i=1}^{n-1} V_i (f \circ \pi) e^{i}  \\
\nabla^{h} f (x, \mathbf{e}) &:= \sum_{i=0}^{n-1}  H_i (f \circ \pi) e^{i}
\end{align*}
Note that $\Vert \nabla^{v} f \Vert^2 = \sum_{i=1}^{n-1} (V_i (f \circ \pi) )^2$ and $\Vert \nabla^{h} f \Vert^2 = \sum_{i=0}^{n-1} (H_i (f \circ \pi))^2 $ and consider also 
\begin{align*}
\Vert \nabla^{\tilde{h} } f \Vert^2 &:= \Gamma^{\tilde{h}} (f) = \sum_{i=1}^{n-1} (H_i (f \circ \pi))^2 \\
\Vert \nabla^{\xi} f \Vert^2 &:= \Gamma^{\xi} (f) = (\xi(f) )^2 = (H_0 (f\circ \pi) ) ^2.
\end{align*}
Thus $\Vert \nabla^{h} f \Vert^2  = \Vert \nabla^{\tilde{h}} f \Vert^2  + \Vert \nabla^{\xi} f \Vert^2$. The underlying metric we use on $T^1\mathcal{M}$ is the Sasaki metric, which makes orthogonal the decomposition $T(T^1 \mathcal{M}) = \mathcal{V} \oplus \tilde{\mathcal{H}} \oplus \mathrm{Vect}(\xi)$.

Introduce also for $i,j =1, \dots, n-1$ the Hessian terms which appears in the computations of the iterated Gamma 
\begin{align*}
\mathrm{Hess}_{i,j}^{v}(f) &:=  V_{i}V_{j}(f \circ \pi) \\ 
\mathrm{Hess}_{i,j}^{v, \tilde{h}} (f) &:= V_i H_j (f \circ \pi) \\ 
\mathrm{Hess}_{i}^{v,\xi} (f) &:= V_i H_0 (f\circ \pi).
\end{align*}
The following lemma gives the explicit expressions of the iterated $\Gamma_2$ for each $\Gamma$ considered (vertical, horizontals, mixed). 
\begin{Lem}\label{gamma2calc}
One has
\begin{align*}
\Gamma^{v}_2 (f) &=\frac{\sigma^2}{2} \left (  \Vert \mathrm{Hess}^v (f)  \Vert^2 + (n-2) \Vert \nabla^{v} f  \Vert^2 \right ) + \kappa \langle \nabla^{v} f, \nabla^{\tilde{h}} f \rangle \\ 
\Gamma^{v,\tilde{h}}_2 (f) &= \frac{\sigma^2}{2} \left ( \langle \mathrm{Hess}^v f , \mathrm{Hess}^{v,\tilde{h}} f    \rangle  + \frac{n-1}{2} \langle  \nabla^{v} f, \nabla^{\tilde{h}} f \rangle  -\langle \mathrm{Hess}^{v,\xi}(f) , \nabla^{v} f  \rangle \right )  \\ 
&\quad +\frac{ \kappa}{2} \left ( \langle R(\nabla^{v} f, e^{0}) e^{0},\nabla^{v} f  \rangle - \Vert   \nabla^{\tilde{h}} f \Vert^2   \right ) \\ 
\Gamma^{\tilde{h}}_2 (f) &= \frac{ \sigma^{2}}{2} \left (  \Vert \mathrm{Hess}^{v, \tilde{h}} f \Vert^2 + \Vert \nabla^{\tilde{h}} f \Vert^2 - 2\langle \mathrm{Hess}^{v,\xi}(f) , \nabla^{\tilde{h}} f  \rangle   \right ) + \kappa  \langle R(\nabla^{v} f, e^{0}) e^{0},\nabla^{\tilde{h}} f  \rangle \\
\Gamma^{\xi}_2 (f) &= \frac{\sigma^2}{2} \left ( \Vert \mathrm{Hess}^{v,\xi} \Vert^2 + (n-1) \Vert \nabla^{\xi} f \Vert^2 + 2 \mathrm{Tr}(\mathrm{Hess}^{v,\tilde{h}}) \xi(f)\right ) 
\end{align*}
\end{Lem}

\begin{proof}
Instead of giving a proof of the computations for any $\Gamma_2$ considered, we first check in all cases the terms involving the vertical Laplacian (i.e the terms appearing beyond $\frac{\sigma^2}{2}$) and we check, in all cases again, the terms involving the vector field $\xi$ (i.e the terms appearing beyond $\kappa$).

Remark that if $\Op$ is a square of some vector field, says $\Op(f):= X^2(f) $ and $\Gamma$ is of the form $\Gamma_f = Y(f) \times Z(f)$ then 
\[
\Gamma_2 (f)  = \frac{1}{2} \left ( [X^2, Y](f) \times Z(f) + Y(f) \times [X^2, Z](f) + 2XY(f) \times XZ(f)  \right ).
\]
in particular if $Y=Z$ one gets, 
\[
\Gamma_2 (f) = XY(f) \times XY(f) + [X^2, Y ](f) \times Y(f).
\]
So in the computation of the $\Gamma_2$, the vertical Laplacian $\sum_i V_i^2 $ in $\Op$ gives 
\begin{itemize}
\item In the computation of $\Gamma^{v}_2$, $\sum_{i,j} V_i V_j (f)  \times V_i V_j (f)+ [V_i^2, V_j] (f)\times V_j(f)$. But $[V_i^2, V_j] = [V_i,V_j]V_i+ V_i[V_i,V_j]= -[V_i,[V_i,V_j] ] + 2V_i[V_i, V_j]$ and since $-[V_i,[V_i,V_j] ] =\mathbf{1}_{i\neq j}V_j$ and $V_i[V_i, V_j] (f\circ \pi) =0 $ on gets summing on $i,j=1,\dots, n-1$
\[
\Vert \mathrm{Hess}^v (f)  \Vert^2 + (n-2) \Vert \nabla^{v} f  \Vert^2.
\]
\item In the computation of $\Gamma^{v,\tilde{h}}_2$,  $\sum_{i,j} V_i V_j(f) \times V_i H_j(f)  + \frac{1}{2} [V_i^2, V_j](f) \times H_j(f)+ \frac{1}{2} V_j(f) \times [V_i^2, H_j](f)$. Here again $[V_i^2, V_j] (f\circ \pi)= \mathbf{1}_{i\neq j}V_j(f\circ\pi)$ and moreover $[V_i^2, H_j] = [V_i, H_j] V_i + V_i [V_i, H_j] =\delta_{ij}([V_i,H_0] - 2 V_iH_0) =  \delta_{ij}(H_i- 2 V_iH_0)$. So summing on $i,j$ one gets
\[
\langle \mathrm{Hess}^v f , \mathrm{Hess}^{v,\tilde{h}} f    \rangle  + \frac{n-2}{2} \langle  \nabla^{v} f, \nabla^{\tilde{h}} f \rangle + \frac{1}{2}  \langle  \nabla^{v} f, \nabla^{\tilde{h}} f \rangle  -\langle \mathrm{Hess}^{v,\xi}(f) , \nabla^{v} f  \rangle.
\]
\item In the computation of $\Gamma^{\tilde{h}}_2$, $\sum_{i,j} V_i H_j (f)\times V_i H_j(f) + [V_i^2, H_j](f) \times H_j(f)$. With $[V_i^2, H_j] = \delta_{ij}(H_i - 2 V_iH_0)$ on gets directly, summing on $i,j$
\[
 \Vert \mathrm{Hess}^{v, \tilde{h}} f \Vert^2 + \Vert \nabla^{\tilde{h}} f \Vert^2 - 2\langle \mathrm{Hess}^{v,\xi}(f) , \nabla^{\tilde{h}} f  \rangle 
\]
\item In the computation of $\Gamma^{\xi}_2$,  $\sum_{i} V_i H_0(f) \times V_i H_0 (f)+ [V_i^2, H_0](f) \times H_0(f)$. And since $[V_i^2, H_0] = [V_i, H_0] V_i + V_i [V_i, H_0] = H_i V_i + V_i H_i = 2 V_i H_i + [H_i,V_i] = 2 V_i H_i + H_0$. Then summing on $i,j$ one obtains (since $H_0 (f \circ \pi)= \xi (f) \circ \pi$)
\[
\Vert \mathrm{Hess}^{v,\xi} \Vert^2 + (n-1) \Vert \nabla^{\xi} f \Vert^2 + 2 \mathrm{Tr}(\mathrm{Hess}^{v,\tilde{h}}) \xi(f)\circ \pi
\]
\end{itemize}
Moreover if  $\Op$ consists of a vector field, $\Op(f):= X(f)$ and $\Gamma$ is of the form $\Gamma(f) =  Y(f) \times Z(f)$ then
\[
\Gamma_2(f) = \frac{1}{2} \left ( [X,Y](f) \times Z(f) + Y(f) \times [X,Z](f) \right ).
\]
Recalling that $[H_0, V_i]= -H_i$, $[H_0, H_j]= \sum_{i=1}^{n-1}\langle R(e^j,e^0)e^0, e^i \rangle V_i$ one gets immediately the $\Gamma_2$ for the $\xi$ part of $\Op$.
\end{proof}

 \subsection{A metric on $T^1\mathcal{M}$ with a a generalized Bakry-\'Emery estimate}
 In this section one proves that one can find a metric on $T^1 \mathcal{M}$ for which there is a generalized Bakry-\'Emery estimate. Our main assumption to obtain this result is 
 
 \vspace{0,5cm}
\noindent \textbf{Assumption A}: We suppose that the Riemannian tensor $R$ of $\mathcal{M}$ is bounded.

\

Using the expression of the different $\Gamma$ obtained in Lemma \ref{gamma2calc} and using extensively the Young inequality one gets the following Proposition:

\begin{Prop}\label{prop.ineq}
There exist $a,b,c,d >0$ with $b^2 < ac$ and $K \in \R$ such that 
\begin{equation}
a \Gamma^v_2(f) -2 b \Gamma^{v,\tilde{h}}_2(f) + c \Gamma^{\tilde{h}}_2(f) +d \Gamma^{\xi}_2(f)  \geq -K \Vert \nabla^{v} f \Vert^2 + \Vert \nabla^{h} f \Vert^2. \label{ineq1}
\end{equation}
\end{Prop}
\begin{proof}
By assumption A, the quantity 
\[
M:= \sup_{x \in \mathcal{M}} \sup_{u,v,w,h \in T^1 \mathcal{M}}  \vert \langle u, R(v,w)h \rangle_{x} \vert,
\]
is finite. 

For any $\ep_1, \dots, \ep_6 >0$ one has, by Young inequality
\begin{align*}
\Gamma^{v}_2 (f) &\geq \frac{\sigma^2}{2} \left (  \Vert \mathrm{Hess}^v (f)  \Vert^2 + (n-2) \Vert \nabla^{v} f  \Vert^2 \right ) -\frac{\kappa}{2} \left ( \frac{1}{\ep_1} \Vert \nabla^{v} f  \Vert^2  + \ep_1 \Vert \nabla^{\tilde{h}} f  \Vert^2   \right ) \\ 
\Gamma^{v,\tilde{h}}_2 (f) &\leq  \frac{\sigma^2}{2} \left ( \langle \mathrm{Hess}^v f , \mathrm{Hess}^{v,\tilde{h}} f    \rangle  + \frac{n-1}{4} \left ( \frac{1}{\ep_2} \Vert \nabla^{v} f  \Vert^2  + \ep_2 \Vert \nabla^{\tilde{h}} f  \Vert^2   \right )   + \frac{1}{2} \left (  \frac{1}{\ep_3} \Vert \nabla^{v} f  \Vert^2  + \ep_3 \Vert \mathrm{Hess}^{v,\xi}(f) \Vert^2 \right) \right )  \\ 
&\quad +\frac{ \kappa}{2} \left ( M \Vert \nabla^{v} f  \Vert^2 - \Vert   \nabla^{\tilde{h}} f \Vert^2   \right ) \\ 
\Gamma^{\tilde{h}}_2 (f) &\geq \frac{ \sigma^{2}}{2} \left (  \Vert \mathrm{Hess}^{v, \tilde{h}} f \Vert^2 + \Vert \nabla^{\tilde{h}} f \Vert^2 - \left (\frac{1}{\ep_4} \Vert \mathrm{Hess}^{v,\xi}(f)  \Vert^2  + \ep_4 \Vert \nabla^{\tilde{h}} f \Vert^2   \right )  \right ) - \frac{\kappa M}{2} \left ( \frac{1}{\ep_5}\Vert \nabla^{v} f \Vert^2 + \ep_5 \Vert \nabla^{\tilde{h}} f  \Vert^2 \right) \\
\Gamma^{\xi}_2 (f) &\geq  \frac{\sigma^2}{2} \left ( \Vert \mathrm{Hess}^{v,\xi} \Vert^2 + (n-1) \Vert \nabla^{\xi} f \Vert^2 - \left ( \frac{n-1}{\ep_6} \Vert \mathrm{Hess}^{v,\tilde{h}} \Vert^2  + \ep_6 \Vert \nabla^{\xi} f  \Vert^2 \right)  \right ) 
\end{align*}
So for $a,b,c,d>0$ one has
\begin{align*}
a \Gamma^v_2(f) -2 b \Gamma^{v,\tilde{h}}_2(f) + c \Gamma^{\tilde{h}}_2(f) +d \Gamma^{\xi}_2(f) &\geq C^v \Vert \nabla^v f\Vert^2 + C^{\tilde{h}} \Vert \nabla^{\tilde{h}}  f  \Vert^2 + C^{\xi} \Vert \nabla^{\xi}  f  \Vert^2 \\
&\quad + A\Vert  \mathrm{Hess}^v (f)  \Vert^2 -2B\langle \mathrm{Hess}^v f , \mathrm{Hess}^{v,\tilde{h}} f    \rangle + C\Vert \mathrm{Hess}^{v, \tilde{h}} f \Vert^2  \\
 &\quad  + D \Vert \mathrm{Hess}^{v,\xi} f \Vert^2
\end{align*}
where
\begin{align*}
C^v &= a \left (\frac{\sigma^2}{2}(n-2) - \frac{\kappa}{2\ep_1} \right) -2b\left (\frac{\sigma^2}{2}\left ( \frac{(n-1)}{4\ep_2}  + \frac{1}{2\ep_3}   \right)  + \frac{\kappa M}{2} \right ) - \frac{c \kappa M}{2 \ep_5} \\
C^{\tilde{h}} &= -a\frac{\kappa}{2} \ep_1 + 2b \left ( \frac{\kappa}{2} -\frac{\sigma^2(n-1)}{8} \ep_2 \right) -c\left (-1+ \frac{\sigma^2}{2} \ep_4 + \frac{\kappa M}{2} \ep_5 \right) \\
C^{\xi} &= d\frac{\sigma^2}{2} \left ( n-1 -\ep_6  \right) \\
A&= \frac{\sigma^2}{2}a \\
B&= \frac{\sigma^2}{2}b \\
C&=\frac{\sigma^2}{2} \left ( c - \frac{n-1}{\ep_6}d  \right)  \\
D&= \frac{\sigma^2}{2} \left ( d- \frac{c}{\ep_4} - b\ep_3 \right)
\end{align*}
We want to choose $a,b,c,d>0$ such that
\[
b^2 < ac, \quad A,B,C,D\geq0, \quad B^2 \leq AC, \quad  C^{\tilde{h}}>0, \ \mathrm{and} \ C^{\xi}>0.
\]
Condition $C\geq 0$ and $C^{\tilde{h}} >0$ implies that necessary one has to choose $d<c$. So fix $\ep>0$ and choose $d= \frac{c}{1+\ep}$ ($c$ will be chosen later). Choose also $\ep_6= (n-1) \frac{2+\ep}{2(1+\ep)}$ and so with this choice: 
\[
C^{\xi}= c \frac{\sigma^2}{2}(n-1) \frac{\ep}{1+\ep}, \quad C= c\frac{\sigma^2}{2} \ep. 
\]
Fix $\ep'>0$ and choose $\ep_3$ such that $b\ep_3 = \frac{c}{\ep_4} \ep'$ so that 
\[
D = \frac{d\sigma^2}{2} \left ( 1- \frac{(1+\ep')(1+\ep)}{\ep_4} \right) 
\]
and fix $\ep_4= (1+\ep')(1+\ep)$ so that $D=0$. In order to choose $C^{\tilde{h}}>0$ we need to impose that $c\frac{\sigma^2}{2}\ep_4$ is small comparing to $b\kappa$. So fix $\ep''>0$ and take $c$ such that $c\frac{\sigma^2}{2}\ep_4 = b \kappa \ep''$ so 
\[
c= b \frac{2\kappa}{\sigma^2} \frac{\ep''}{(1+\ep)(1+\ep')}.
\]
Choose now $\ep_5,\ep_2$ and $\ep_1$ such that $ c \frac{\kappa M}{2} \ep_5 = b \kappa \ep''$, $\frac{\sigma^2(n-1)}{4}\ep_2 = \kappa \ep''$ and $\frac{a\kappa}{2} \ep_1 = b\kappa \ep''$. Thus 
\[
C^{\tilde{h}}= b \kappa \left (1 - 4\ep'' +\frac{2 \ep''}{\sigma^2(1+\ep)(1+\ep')} \right)
\]
and the condition $C^{\tilde{h}} >0$ is satisfied if, for example, $\ep'' < \frac{1}{4}$.
Moreover, the condition $B^2 \leq AC$ is equivalent to $b^2 \leq ac\ep$ (which implies, if $\ep<1$, $b^2<ac$). Thus assume $\ep<1$ and choose $a= \frac{b^2}{c\ep}$ thus
\[
a= b \frac{\sigma^2 (1+\ep)(1+\ep')}{2 \kappa \ep \ep'' }.
\]
Let fix now $\ep''$ such that $C^{\xi}=C^{\tilde{h}}$, for that take 
\[
\ep'' = \frac{1}{4} \times \frac{1}{1+ \frac{(n-1)\ep}{4(1+\ep')}}.
\]
Finally fix $b$ such that $C^{h}=C^{\xi}=1$ and one obtains 
\[
b= \frac{1}{\kappa} \left ( 1+ \frac{4(1+\ep')}{(n-1)\ep}  \right ). 
\]
To sum up, one can describe all the obtained parameters $a,b,c$ and $d$ by the mean of two parameters $\ep \in ]0,1[$ and $\ep'>0$
\begin{align*}
a&= \frac{\sigma^2}{2 \kappa^2} (n-1) \times \left ( 1+ \frac{4(1+\ep')}{(n-1)\ep} \right )^2 \times (1+\ep) \\
b&= \frac{1}{\kappa} \left ( 1+ \frac{4(1+\ep')}{(n-1)\ep} \right ) \\
c&= \frac{2}{\sigma^2} \times \frac{1}{n-1} \times  \frac{\ep}{1+\ep} \\
d&=  \frac{2}{\sigma^2} \times \frac{1}{n-1} \times  \frac{\ep}{(1+\ep)^2}.
\end{align*}
Finally the expression (quite complicated) of $C^{v}$ can be also given in terms of $\ep$ and $\ep'$. We just gives here the equivalent when $\kappa=\sigma$ and goes to infinity
\[
C^{v} \underset{\sigma \to +\infty}{\sim} \frac{\sigma^2}{2} \times K_{\ep,\ep'}
\]
where $K_{\ep,\ep'}$ is the constant 
\[
\frac{n-1}{2} \left ( 1+ \frac{4(1+\ep')}{(n-1)\ep} \right )^2  (1+\ep)  \left (  n-2 - \left (4+ \frac{\ep^2}{\ep'} \right ) \frac{(1+\ep)(1+\ep')}{\ep} - \frac{(n-1)^2}{16} \times \frac{\ep}{(1+\ep)(1+\ep')}\right  )
\]
\end{proof}

Define now the tensor $\mathcal{T}$ and its iterated $\mathcal{T}_2$ as 
\begin{align*}
\mathcal{T}(f) &:= a \Vert \nabla^v f \Vert^2 -2b \langle \nabla^v f, \nabla^{\tilde{h}} f   \rangle + c \Vert \nabla^{\tilde{h}}f \Vert^2 + d \Vert \nabla^\xi f  \Vert^2\\
\mathcal{T}_2(f)&:=\frac{1}{2} \left ( \Op ( \mathcal{T}(f) ) -2 \mathcal{T}(f, \Op f)  \right ) \\ 
&= a \Gamma^v_2(f) -2 b \Gamma^{v,\tilde{h}}_2(f) + c \Gamma^{\tilde{h}}_2(f) +d \Gamma^{\xi}_2(f).
\end{align*}
where the coefficient $a,b,c$ and $d$ are the one of the previous Proposition. 

Note that the \emph{carr\'e du champ} operator $\Gamma$ associated to  $\Op$ is 
\begin{align*}
\Gamma (f) &= \frac{\sigma^2}{2} \Gamma^{v}(f).
\end{align*}

Moreover, since $\mathcal{T}(f) \leq (a+ b) \Vert \nabla^v f \Vert^2  + \max(b+c, d) \Vert \nabla^h f \Vert^2$, inequality \eqref{ineq1} of the previous Proposition gives the following generalized Bakry-\'Emery estimate for $\Op$: 
\begin{align}
\mathcal{T}_2 (f) \geq \rho \mathcal{T}(f) - K \Gamma(f) \label{ineq2}
\end{align}
where
\[
\rho = \frac{1}{\max(b+c, d)}, \quad \mathrm{and} \quad  K= \left  (-C^v + \frac{a+b}{\max(b+ c, d)} \right ) \frac{2}{\sigma^2}.
\]
\begin{Remark}
 Observe that when $\kappa = \sigma$ and $k,\sigma$ both go to infinity we have 
 \begin{equation}
K \underset{\sigma \to +\infty}{\longrightarrow} -K_{\ep,\ep'}. \label{Kasympt}
\end{equation}
where $K_{\ep, \ep'}$ is defined at the end of Proposition  \ref{prop.ineq}. 

\end{Remark}

\subsection{Convergence to equilibrium in $H^1$ }

Denote by $\mu$ the volume on $T^1\mathcal{M}$ for the Sasaki metric. For $f: T^1 \mathcal{M} \to \R$, bounded one has: 
\[
\int_{T^1 \mathcal{M}} f d \mu = \int_{\mathcal{M}} \left ( \int_{T^1_x \mathcal{M} } f(x, v) \mathrm{Vol}_{T^1_x \mathcal{M}}(dv)    \right ) \mathrm{Vol}_{\mathcal{M}}(dx).
\]
One supposes that the volume $\mathrm{Vol}_{\mathcal{M}}$ is finite, and thus $\mu$ is finite too. We denote by $P_t$ the semigroup generated by $\Op$, that is
\[
\forall (x,v) \in T^1 \mathcal{M}, \quad P_t f (x) = \mathbb{E}_{(x,v)} [ f(x_t, v_t) ],
\]
where $(x_t,v_t)_{t \geq 0}$ is the velocity spherical Brownian motion, i.e. the process generated by $\mathcal{L}$. Since the manifold is assumed  to be complete, the lifetime of this process is infinite. Using the  generalized Bakry-\'Emery estimate \eqref{ineq2} one deduces  the following exponential convergence in the $H^1$ norm (as in \cite{baudoin-bakry}, \cite{BaudoinEMS}).

\begin{Prop}\label{PropH1}
Suppose that $\mu$ satisfies a Poincar\'e inequality (for the Sasaki metric)
\[
\int f^2 d\mu - \left ( \int f d\mu \right)^2   \leq  \frac{1}{\lambda} \int \Vert \nabla^v f \Vert^2 +  \Vert \nabla^{h} f \Vert^2d \mu 
\]
then one has the following exponential convergence to equilibrium in the  $H^1(\mu,\mathcal{T})$ norm 
\begin{align}
\Vert P_t f - \mu(f) \Vert^2_{H^1(\mu,\mathcal{T})} \leq e^{- 2\tilde{\lambda}t } \Vert f- \mu(f) \Vert^2_{H^1(\mu,\mathcal{T})}  \label{CVH1}
\end{align}
where 
\[
\Vert f \Vert^2_{H^1(\mu,\mathcal{T})} := \int f^2 + K \mathcal{T}(f) d \mu
\]
and $\tilde{\lambda}$ can be explicitly written in term of $\lambda$, $K$ and $\rho$.
\end{Prop}
\begin{proof}
We note  that the law of the velocity spherical Brownian motion $(x_t, v_t)$ starting at $(x_0,v_0)\in T^{1}\mathcal{M}$ is supported by the compact $\{ (x,v) \in T^{1}\mathcal{M}, \quad d(x,x_0) \leq t \}$. Thus if $f$ is smooth and compactly supported then so is $P_t f $. For that reason the following computations are well justified. 

Define  $\Lambda_s := P_{t-s}( \mathcal{T}(P_{s}f  ) ) + K P_{t-s} ( (P_{s} f)^2)$ which is compactly supported when $f$ is. Then, using inequality \eqref{ineq2},
\[
\Lambda'_s = -2 P_{t-s}(\mathcal{T}_2 (P_{s}f)) - 2 KP_{t-s} (\Gamma(P_s f)) \leq -2 \rho P_{t-s} (\mathcal{T}(P_s f)).
\]
By integrating with respect to $\mu$, one obtains for some $\eta >0$
\[
\int \Lambda'_s d\mu \leq -2 \rho \eta \left (  \int \mathcal{T}(P_s f) d\mu  + \frac{1-\eta}{\eta}  \int \mathcal{T}(P_s f) d\mu  \right  ).
\]
Now, since $\mathcal{T}(f) \geq  \min(a(1-\sqrt{\ep}), c(1-\sqrt{\ep}),d) \Vert \nabla f \Vert^2$, the Poincar\'e inequality can be written, for $f$ such that $\int f d \mu =0$,  
\[
\frac{1}{\hat{\lambda}} \int \mathcal{T}(P_s f) d\mu  \geq \int (P_s f)^2 d \mu
\]
where $\hat{\lambda} := \lambda \times \min(a(1-\sqrt{\ep}), c(1-\sqrt{\ep}),d)$.

Thus, by choosing $\eta$ such that $\frac{1-\eta}{\eta}  =\frac{K}{ \hat{\lambda}}$ i.e $\eta = \frac{\hat{\lambda}}{K + \hat{\lambda}}$ we have
\[
\int \Lambda'_s  d\mu \leq -2  \rho \eta \int \Lambda_s d\mu
\]
and by Gronwall inequality, 
\[
\int \Lambda_t d\mu \leq e^{-2 \rho \eta } \int \Lambda_0 d\mu,
\]
which is 
\[
\Vert P_t f  \Vert_{H^1} \leq e^{- 2\tilde{\lambda}t } \Vert f \Vert_{H^1}
\]
for $\tilde{\lambda} =  \rho \eta = \frac{\rho \hat{\lambda}}{K + \hat{\lambda}}$. 
\end{proof} 

\begin{Remark}
When $\kappa= \sigma$ and $\sigma$ goes to infinity one obtains (recall that $\rho := (\max(b+c,d))^{-1}$ and  $\hat{\lambda}:= \lambda  \times \min(a(1-\sqrt{\ep}), c(1-\sqrt{\ep}),d)$  ) that $\rho$ is of order $\sigma$ and $\hat{\lambda}$ of order $1/\sigma^2$. So the rate $\tilde{\lambda}$ obtained goes to zero when $\sigma$ goes to infinity.
\end{Remark}

\begin{Remark}
One can provide a sufficient condition on $\mathcal{M}$ only so that the Poincar\'e inequality on $T^1\mathcal{M}$ is satisfied. Indeed, assume that the Ricci curvature of $\mathcal{M}$ (for the Levi-Civita connection) is a Codazzi tensor, that is for any smooth vector fields $X,Y,Z$ on $\mathcal{M}$,
\[
(\nabla_X \text{Ric}) (Y, Z)=(\nabla_Y \text{Ric}) (X, Z).
\]
Then, the Bott connection associated to the totally geodesic Riemannian submersion $T^1\mathcal{M} \to \mathcal{M}$ is of Yang-Mills type (see Example 4.4 in \cite{BaudoinEMS}). Moreover, the boundedness of the Riemann curvature tensor implies the boundedness of torsion tensor of the Bott connection.  As a consequence, if the integrated Ricci curvature of $\mathcal{M}$ is bounded from below by a positive constant, one deduces that the horizontal Laplacian on $T^1\mathcal{M}$ satisfies an integrated generalized curvature-dimension inequality in the sense of \cite{BG}. As such, the horizontal Laplacian has a spectral gap (see \cite{BaudoinEMS}). Since the horizontal Dirichlet form is dominated by the total Dirichlet form on $T^1\mathcal{M}$, one concludes that the Poincar\'e inequality on $T^1\mathcal{M}$ is satisfied.
\end{Remark}

\subsection{Convergence to equilibrium in $L^2$}

In order to get the convergence in $L^2$ we need a regularization result which controls the gradient of $P_t f$ by the $L^2$ norm of $f$.  To get this inverse Poincar\'e type inequality we adapt H\'erau's method \cite{Herau} and look  now for   coefficients $a,b,c$ and $d$  as in Proposition \ref{prop.ineq}, but depending now on time $t$ such that a positive expression involving the horizontal and vertical gradients is decreasing with $t$.  

\subsubsection{$\Sigma$-calculus}

As observed by Gadat and Miclo in \cite{GM13}, even if in the literature about hypocoercivity, the authors consider mostly brackets of first order operators,  it seems that in some models, the keypoint is given by brackets between second order operators. In our case of the velocity spherical Brownian motion, the $3$-order Lie bracket $[V_i, [V_i, H_0]]$, necessary to reach the direction $H_0$ (or $\xi$), appears naturally in the bracket $[ \Delta^v, \xi]$ that we already met in the computation of the $\Gamma^{\xi}_2$ (it is exactly what gives the term $\Vert \nabla^{\xi} f \Vert^2$ in $\Gamma_2^\xi$). However it seems to be not enough to get the regularization result that we are looking for. Instead, we remark that the bracket $[ \Delta^v, \xi ]$ appears also when we consider the "$\Gamma_2$" of the square of the vertical Laplacian. This is quite close to some $\Gamma_3$ computation and makes appear $3$rd-order derivatives.

Let introduce $\Sigma^{v} (f) := \left ( \Delta^v f \right )^2$, $\Sigma^{v, \xi} (f) := \Delta^v f \times \xi(f)$ and 
\begin{align*}
\Sigma_2^v (f) &= \frac{1}{2} \left ( \Op(\Sigma^v (f )) - 2 \Sigma^v ( \Op f, f )   \right )    \\
\Sigma_2^{v,\xi} (f) &= \frac{1}{2} \left ( \Op (\Sigma^{v,\xi} (f) ) - \Sigma^{v,\xi} (\Op f, f ) - \Sigma^{v,\xi} ( f, \Op f )  \right ).
\end{align*}

A straightforward, if not tedious, computation yields

\begin{Lem} \label{sigma}
\begin{align*}
\Sigma_2^v (f) &= \frac{\sigma^2}{2}  \Vert  \nabla^v ( \Delta^v f )   \Vert^2 + \kappa  [ \xi, \Delta^v ] f \Delta^v f \\
 &= \frac{\sigma^2}{2}  \Vert  \nabla^v ( \Delta^v f )   \Vert^2 - (n-1)\kappa \Delta^v f \times \xi(f) - 2\kappa \Delta^v f \times \Delta^{v, \tilde{h}}f   \\
\Sigma_2^{v,\xi} (f) &= \frac{\sigma^2}{4} \left ( \Delta^v f \times [ \Delta^v, \xi ] f + 2\langle \nabla^v \Delta^v f , \mathrm{Hess}^{v,\xi} f\rangle   \right)  + \frac{\kappa}{2} [ \xi, \Delta^v ] f \xi(f) \\
&=  \frac{(n-1)\sigma^2}{4} \Delta^v f \times \xi(f) + \frac{\sigma^2}{2} \Delta^v f \times \Delta^{v, \tilde{h}}f  +    \frac{\sigma^2}{2}\langle \nabla^v \Delta^v f , \mathrm{Hess}^{v,\xi} f  \rangle \\
& \quad \quad - \frac{(n-1)\kappa}{2} \Vert \nabla^{\xi} f \Vert^2 - \kappa \Delta^{v, \tilde{h}}f \times \xi(f)
\end{align*}
\end{Lem}

\subsubsection{Regularization}
Let $f$ be a smooth function on $T^1 \mathcal{M}$ and $t>0$ fixed. For $s\in [0,t]$, denote by 
\begin{align*}
\Lambda(s) := P_{t-s} \left ( \left ( P_s f   \right )^2   \right ), \ \ \Gamma^{v} (s) := P_{t-s} &\left ( \Gamma^{v}(P_s f )   \right ), \ \ \Gamma^{v,\tilde{h}} (s) := P_{t-s} \left ( \Gamma^{v,\tilde{h}}(P_s f )   \right ), \\
 \Gamma^{\tilde{h}} (s) := P_{t-s} \left ( \Gamma^{\tilde{h}}(P_s f )   \right ), \  &\Gamma^{\xi} (s) := P_{t-s} \left ( \Gamma^{\xi}(P_s f )   \right ), 
\end{align*}
and also
\begin{align*}
 \Gamma_2^{v} (s) &:= P_{t-s} \left ( \Gamma_2^{v}(P_s f )   \right ), \  \Gamma_2^{v,\tilde{h}} (s) := P_{t-s} \left ( \Gamma_2^{v,\tilde{h}}(P_s f )   \right ), \\
 \Gamma_2^{\tilde{h}} (s) &:= P_{t-s} \left ( \Gamma_2^{\tilde{h}}(P_s f )   \right ), \  \Gamma_2^{\xi} (s) := P_{t-s} \left ( \Gamma_2^{\xi}(P_s f )   \right ).
\end{align*}

Introduce also (with a slight abuse of notation) 
\begin{align*}
\Sigma^v (s) &:= P_{t-s} \left ( \Sigma^v (P_s f)  \right ) \\
\Sigma^v_2 (s) & := P_{t-s} \left ( \Sigma_2^v (P_s f)  \right )
\end{align*}
For $a,b,c, \hat{a}, \hat{b}, \hat{c}$ non negative numbers we introduce for $s\in [0,t]$
\begin{align*}
\mathcal{F}_s := \Lambda(s) + as^2 \Gamma^{v} (s) -2 b s^4 \Gamma^{v,\tilde{h}} (s) + c s^6 \Gamma^{\tilde{h}} (s)  + \hat{a}s^4 \Sigma^v (s) - 2 \hat{b} s^6 \Sigma^{v,\xi}(s)  + \hat{c} s^8 \Gamma^{\xi}(s).
\end{align*}
Then 
\begin{align*}
\frac{d}{ds} \mathcal{F}_s &= - 2 \Gamma (s) + 2as  \Gamma^v (s) - 2as^2 \Gamma_{2}^{v}(s) - 8bs^3\Gamma^{v,\tilde{h}} (s) + 4 bs^4 \Gamma_{2}^{v,\tilde{h}} (s) + 6cs^5 \Gamma^{\tilde{h}}(s) - 2c s^6 \Gamma_2^{\tilde{h}} (s) \\ 
& + 4 \hat{a} s^3 \Sigma^v (s) -2 \hat{a} s^4 \Sigma_{2}^v (s) - 12 \hat{b} s^5 \Sigma^{v, \xi} (s) +4\hat{b}s^6 \Sigma_2^{v,\xi}(s) + 8 \hat{c} s^7 \Gamma^{\xi}(s) -2 \hat{c} s^8 \Gamma_2^\xi (s). 
\end{align*}

\begin{Lem}\label{regularization}
One can find non negative $a,b,c, \hat{a}, \hat{b}, \hat{c} $ such that $b^2 < ac$ and $\hat{b}^2 < \hat{a} \hat{c}$ such that for $0\leq s \leq t\leq t_0$ one has
\[
\frac{d}{ds} \mathcal{F}_s \leq 0.
\]
So 
\[
\mathcal{F}_s \leq \mathcal{F}_0 = P_t (f^2). 
\]
This implies in particular that one can find $\tilde{a}, \tilde{c}>0$ such that for $t\leq t_0$ 
\begin{align*}
\Vert \hat{\nabla}^v P_t f \Vert^2 \leq \frac{\tilde{a}}{t^2}P_t \left (f^2 \right), \ \mathrm{and}, \  \Vert \hat{\nabla}^{h} P_t f \Vert^2 \leq \frac{\tilde{c}}{t^8}P_t \left (f^2 \right). 
\end{align*}
\end{Lem}
\begin{proof}
One takes inequalities (involving $\ep_1, \dots, \ep_5$) on the $\Gamma_2$ used in the proof of Proposition  \ref{prop.ineq}. Moreover, by lemma \ref{sigma} one has for $\ep_7, \dots, \ep_{11} >0$  
\begin{align*}
\Sigma_2^v(f) &\geq  \frac{\sigma^2}{2} \Vert \nabla^v \Delta^v f  \Vert^2 - \frac{\kappa(n-1)}{2} \left ( \ep_7 \Vert  \nabla^\xi f \Vert^2 + \frac{n-1}{\ep_7} \Vert  \mathrm{Hess}^v f\Vert^2  \right) - \kappa (n-1) \left ( \ep_8 \Vert \mathrm{Hess}^{v,\tilde{h}}  \Vert^2 + \frac{1}{\ep_8} \Vert \mathrm{Hess}^{v}  \Vert^2  \right) \\
\Sigma_2^{v,\xi} (f) & \leq - \frac{(n-1)\kappa}{2} \Vert \nabla^{\xi} f \Vert^2 + \frac{\sigma^2}{2} \langle \nabla^v \Delta^v f , \mathrm{Hess}^{v,\xi} f \rangle + \frac{(n-1)\sigma^2}{8} \left ( \frac{1}{\ep_9}(\Delta^v f)^2  + \ep_9 \Vert \nabla^\xi f \Vert^2\right )  \\
& \quad \quad  +  \frac{\sigma^2}{4} \left ( \frac{1}{\ep_{10}}(\Delta^v f)^2  + (n-1) \ep_{10} \Vert \mathrm{Hess}^{v,\tilde{h}} \Vert^2 \right ) + \frac{\kappa}{2} \left ( \ep_{11} \Vert \nabla^{\xi} f \Vert^2 + \frac{n-1}{\ep_{11}} \Vert \mathrm{Hess}^{v,\tilde{h}} \Vert^2\right).
\end{align*}
We have also for $\ep_{12}, \ep_{13}>0$ by Young inequality
\begin{align*}
 \Sigma^{v,\xi} f   &\leq  \frac{1}{2} \left ( \frac{n-1}{\ep_{12}} \Vert \mathrm{Hess}^v f \Vert^2  + \ep_{12} \Vert \nabla^{\xi} f  \Vert^2 \right ) \\
 \Gamma^{v, \tilde{h}} f  &\leq \frac{1}{2} \left ( \frac{1}{\ep_{13}} \Vert \nabla^v f \Vert^2  + \ep_{13} \Vert \nabla^{\tilde{h}} f  \Vert^2 \right ).
\end{align*}
Then 
\begin{align*}
\frac{d}{ds} \mathcal{F}_s &\leq A \Vert \mathrm{Hess}^{v} P_s f \Vert^2 - 2B \langle \mathrm{Hess}^{v} , \mathrm{Hess}^{v,\tilde{h}} P_s f  \rangle + C \Vert \mathrm{Hess}^{v, \tilde{h}} P_s f \Vert^2  \\
& \quad \hat{A} \Vert \nabla^v \Delta^v P_s f \Vert^2 -2\hat{B} \langle \nabla^v \Delta^v P_s f , \mathrm{Hess}^{v,\xi} P_s f \rangle + \hat{C} \Vert \mathrm{Hess}^{v,\xi} P_s  f \Vert^2 \\ 
& \quad + C^v \Vert \nabla^v P_s \Vert^2 + C^{\tilde{h}} \Vert  \nabla^{\tilde{h}} P_s f \Vert^2 + C^{\xi} \Vert \nabla^{\xi} P_s f \Vert^2. 
\end{align*}
where 
\begin{align*}
A&= -2 a s^2 \frac{\sigma^2}{2} + 4 \hat{b} s^6 \frac{(n-1)^2 \sigma^2}{8} \times \left ( \frac{1}{\ep_9} + \frac{2}{n-1} \frac{1}{\ep_{10}}\right )  \\ 
& \quad + 4 \hat{a} s^3 (n-1) + 2 \hat{a} s^4 \left ( \frac{\kappa (n-1)^2}{2} \frac{1}{\ep_7} + \frac{\kappa (n-1)}{\ep_8}  \right ) + 6 \hat{b} s^5 \frac{n-1}{\ep_{12}} \\
 B&=-2b s^4 \frac{\sigma^2}{2} \\
 C&= -2 c s^6 \frac{\sigma^2}{2} + 2 \hat{a} s^4 \kappa (n-1) \ep_8 + 4 \hat{b} s^6 (n-1) \ep_{10} \frac{\sigma^2}{4} 
 + 4 \hat{b} s^6 \frac{\kappa}{2} \frac{n-1}{\ep_{11}} + 2 \hat{c} s^8 \frac{\sigma^2}{2} \times \frac{n-1}{\ep_6} \\ 
 \hat{A}&= -2 \hat{a} s^4 \frac{\sigma^2}{2} \\ 
 \hat{B}&= -2 \hat{b} s^6 \frac{\sigma^2}{2} \\ 
 \hat{C} &= 4b s^4 \frac{\sigma^2}{2} \times \frac{\ep_3}{2} + 2 c s^6 \frac{\sigma^2}{2} \times \frac{1}{\ep_4} -2 \hat{c} s^8 \frac{\sigma^2}{2} \\
 C^v &= - \sigma^2 + 2as -2as^2 \left ( \frac{\sigma^2}{2} (n-2) - \frac{\kappa}{2} \times \frac{1}{\ep_1} \right ) + \frac{4bs^3}{\ep_{13}} \\
  &\quad + 4bs^4 \left ( \frac{n-1}{4} \times \frac{\sigma^2}{2} \times \frac{1}{\ep_2} + \frac{1}{2\ep_3}+ \frac{\kappa}{2} M \right )  +2 cs^6 \times \frac{\kappa M}{2} \times \frac{1}{\ep_5} \\
 C^{\tilde{h}}&= 2as^2 \frac{\kappa}{2} \ep_1 + 4bs^4 \left ( -\frac{\kappa}{2} + \frac{n-1}{4} \ep_2 \right ) + 4 b s^3 \ep_{13} + 6cs^6 \left ( \frac{\sigma^2}{2} - \frac{\sigma^2}{2} \ep_4 - \frac{\kappa M}{2} \ep_5  \right ) \\ 
 C^{\xi} &= 2\hat{a}s^4 \frac{\kappa (n-1)}{2} \ep_7 + 6 \hat{b} s^5 \ep_{12} + 8 \hat{c} s^7 -2 \hat{c} s^8 \frac{\sigma^2}{2}\left ( n-1 - \ep_6  \right )  - 4 \hat{b} s^6 \left ( \frac{(n-1)\kappa}{2} - \frac{(n-1)}{8} \sigma^2 \ep_9 - \frac{\kappa}{2}\ep_{11}  \right ) 
\end{align*}
We want to choose $\ep_1, \dots, \ep_{13}$, $a,b,c$ and $\hat{a}, \hat{b}, \hat{c}$ such that 
\begin{align*}
&b^2 < ac, \ \hat{b} < \hat{a} \hat{c}, \   B^2 \leq AC, \  \hat{B}^2 \leq \hat{A} \hat{C},  \ C^v \leq 0, \ C^{\tilde{h}} \leq 0, \ C^{\xi} \leq 0 \\ 
&A \leq 0, \quad B \leq 0, \quad C \leq 0, \quad \hat{A} \leq 0, \quad \hat{B} \leq 0, \quad \hat{C} \leq 0.
\end{align*}
First fix $\ep_3:= \frac{\hat{c}}{4b} s^4$ and $\ep_4 = \frac{4c}{\hat{c}} \frac{1}{s^2}$ and so 
\[
\hat{C}= -\hat{c} s^8 \frac{\sigma^2}{2}.
\]
Then fix $\ep_5 =1$, $\ep_{13}=\frac{\kappa}{8}s$, $\ep_1 = \frac{b}{2a} s^2$, $\ep_2= \frac{\kappa}{2(n-1)}$ and so 
\[
C^{\tilde{h}} = \left ( -b \frac{\kappa}{2} + 12 \frac{c^2}{\hat{c}}\sigma^2  \right )s^4 + o(s^4).
\] 
Fix also $\ep_9 =s$, $\ep_6=1$, $\ep_{11}= \frac{n-1}{4}$, $\ep_7 = \frac{\hat{b}}{2 \hat{a}}s^2$, $\ep_{12}=\frac{(n-1)\kappa}{12}s$ and so 
\[
C^{\xi} = - \hat{b} \frac{(n-1)\kappa}{2} s^6 + o(s^6).
\]
Finally fix $\ep_{10}=s$ and $\ep_{8}= \frac{c\sigma^2}{4\hat{a} \kappa (n-1)}s^2$ so that 
\[
C= \left (-c \frac{\sigma^2}{2} +16 \hat{b} \frac{\kappa}{2}  \right ) s^6 + o(s^6).
\]
We have also 
\begin{align*}
A&= \left ( -2a \frac{\sigma^2}{2} + \frac{(2 \hat{a})^2}{\hat{b}} \frac{\kappa(n-1)^2}{2} + 8 \frac{\hat{a}}{c} \frac{\kappa^2 (n-1)^2}{\sigma^2} \right )s^2 + o(s^2)\\
C^v &= -\sigma^2 + \frac{2a^2 \kappa}{b} + \frac{16b^2}{2\hat{c}} + o(1).
\end{align*}
Then choose $\hat{b}$ such that $16\hat{b} \frac{\kappa}{2} = \frac{1}{2} \times c \frac{\sigma^2}{2}$ so finally
\[
C=  -\frac{c}{2} \frac{\sigma^2}{2}  s^6 + o(s^6).
\]
Then choose $\hat{a}$ such that $\frac{(2 \hat{a})^2}{\hat{b}} \frac{\kappa(n-1)^2}{2} +8 \frac{\hat{a}}{c} \frac{\kappa^2 (n-1)^2}{\sigma^2} =  a \frac{\sigma^2}{2}$ so that 
\[
A= -a \frac{\sigma^2}{2} s^2 + o(s^2). 
\]
Then the condition $B^2 \leq A C$ is satisfied, when $s$ is sufficiently small, by choosing $a,b,c$ such that
$4b^2 < ac$. 
We finish by choosing $\hat{c}$ sufficiently large so that $12 \frac{c^2}{\hat{c}}\sigma^2 < -b \frac{\kappa}{2}$  (so $C^{\tilde{h}} >0$) and so that $2\hat{b}^2 < \hat{a}\hat{c}$ (so $\hat{B}^2 \leq \hat{A} \hat{C}$. For  the condition $C^v >0$ be satisfied we take $a$ and $b$ such that $ \frac{2a^2 \kappa}{b} = \frac{\sigma^2}{3}$ and take also $\hat{c}$  sufficiently large so that $ \frac{16b^2}{2\hat{c}} \leq \frac{\sigma^2}{3}$. 

\end{proof}

\subsubsection{Convergence in $L^2$}
This is a consequence of the convergence in $H^1$ and the previous regularization lemma.
\begin{Prop}
Under the hypotheses of Proposition \ref{PropH1}, there is $C>1$ such that for all function $f \in L^2 (\mu)$ such that $\int f d\mu =0$ one has for all $t>0$
\[
\Vert P_t f \Vert^2_{H^1(\mu,\mathcal{T})} \leq C e^{-2\tilde{\lambda}t} \Vert f \Vert^2_{L^2(\mu)}.
\]
\end{Prop}
\begin{proof}
As in the proof of Proposition \ref{PropH1} one obtains
\[
\Vert  P_t f \Vert_{H^1}^2 \leq e^{-2\tilde{\lambda}(t-t_0)} \Vert P_{t_0} f \Vert_{H^1}^2,
\]
and by Lemma \ref{regularization}, one can find a constant $\tilde{C}>0$ such that
\[
 \int K \mathcal{T}(P_{t_0}(f)) d\mu \leq \frac{\tilde{C}}{t_{0}^8} \int  f^2 d\mu. 
\]
Moreover, since $t \mapsto \int (P_t f )^2 d \mu$ is decreasing, we have also $\int (P_{t_0} f )^2 d \mu \leq \int  f^2 d\mu$ 
\end{proof}

\end{document}